\documentclass[12pt]{amsart}
\usepackage[T1]{fontenc}
\usepackage[utf8]{inputenc}
\usepackage{color,amssymb,amsmath,mathrsfs,enumerate,esint}
\usepackage{a4wide}
\usepackage{hyperref}
\theoremstyle{plain}
\newtheorem{Theorem}{Theorem}[section]
\newtheorem{Lemma}[Theorem]{Lemma}

\theoremstyle{definition}
\newtheorem{Remark}[Theorem]{Remark}

\title[Proof of Gagliardo-Nirenberg inequality with historical remarks]{Detailed proof of classical Gagliardo-Nirenberg interpolation inequality with historical remarks}
\author{A. Fiorenza, M.R. Formica, T. Roskovec and F. Soudsk\' y}
\address{Universit\`{a} di Napoli Federico II, Dipartimento di Architettura, via Monteoliveto, 3, 80134 - Napoli (Italy) and Consiglio Nazionale delle Ricerche, Istituto per le Applicazioni del Calcolo Mauro Picone, Sezione di Napoli, via Pietro Castellino, 111, 80131 - Napoli (Italy)}
\email{fiorenza@unina.it}
\address{Universit\`{a} di Napoli Parthenope, via Generale Parisi, 13, 80132 - Napoli (Italy)}
\email{mara.formica@uniparthenope.it}
\address{Faculty of Economics, University of South Bohemia, Studentsk\' a 13, \v Cesk\' e Bud\v ejovice, Czech Republic and Faculty of Information Technology, Czech Technical University in Prague, Th\'akurova 9, 160 00 Prague 6, Czech Republic }
\email{troskovec@ef.jcu.cz}
\address{Faculty of Economics, University of South Bohemia, Studentsk\' a 13, \v Cesk\' e Bud\v ejovice, Czech Republic}
\email{fsoudsky@ef.jcu.cz}
\subjclass[2010]{46E35, 35A23, 26D10}
\keywords{Gagliardo-Nirenberg inequality, interpolation inequality, intermediate derivatives, Sobolev spaces, Sobolev embedding theorem, inequalities for derivatives}
\begin{document}
\begin{abstract}A carefully written Nirenberg's proof of the well known Gagliardo-Nirenberg interpolation inequality for intermediate derivatives in $\mathbb{R}^n$ seems, surprisingly, to be missing in literature. In our paper we shall first introduce this fundamental result and provide information about it's historical background. Afterwards we present a complete, student-friendly proof. In our proof we use the architecture of Nirenberg's proof, the proof is, however, much more detailed, containing also some differences. The reader can find a short comparison of differences and similarities in the final chapter.
\end{abstract}
\maketitle
\section{The result}
The most general Gagliardo-Nirenberg inequality for intermediate derivatives in $\mathbb{R}^n$, as known from the original papers by E. Gagliardo (\cite{Gag2}) and L. Nirenberg (\cite{Nire1}), can be stated as follows:
\begin{Theorem}\label{MAINR}
Let $1\le q\le \infty$ and $j,k\in \mathbb{N}$, $j<k$, and either
\begin{equation}\label{paramet}
\begin{cases}
r=1\\
\displaystyle\frac{j}{k}\le \theta\le 1
\end{cases}
or \quad
\begin{cases}
1<r<\infty\\
k-j-\displaystyle\frac{n}{r}=0,1,2,\ldots\\
\displaystyle\frac{j}{k}\le \theta< 1
\end{cases}\, .
\end{equation}
If we set
\begin{equation}\label{GNGE}
\frac{1}{p}=\frac{j}{n}+\theta\left(\frac{1}{r}-\frac{k}{n}\right)+\frac{1-\theta}{q}\, ,
\end{equation}
then there exists constant $C$ independent of $u$ such that
\begin{equation}\label{GGNII}
\|\nabla^{j}u\|_{p}\leq C\|\nabla^{k}u\|_{r}^{\theta}\|u\|_{q}^{1-\theta}\qquad\forall u\in L^q(\mathbb{R}^n)\cap W^{k,r}(\mathbb{R}^n)\, .
\end{equation}
\end{Theorem}

The case $\theta=1$ corresponds to the Sobolev inequality. If $\theta=1$ and $k-j-n/r$ is a nonnegative integer, the inequality does not hold. It would correspond to the critical cases of the Sobolev inequality, where some estimates for derivatives are still true, but leave the scale of Lebesgue spaces. We omit further comments on this classical result. The reader is referred to the original paper by L. Nirenberg \cite{Nire1} and the wide literature, classical and recent, on this well-known result (see e.g. \cite{Brezisuni,Maz,AdaFou,Nire2,Frie,Lady,Kjf,GilTru,Zie,EvaGar}).

It would be impossible to describe in a satisfactory way the popularity of Theorem \ref{MAINR} in the Sobolev spaces theory and its applications in the PDEs theory. The development of literature connected with inequality \eqref{GGNII} in the past 60 years has been really impressive because of the several variants and generalizations considered (see e.g. \cite{Solon,Bro2,Golo,CapFioKala,CriMare,Giannetti,KalaKrb,BrezisGaNi,Triebel, Strz, McCRoRo,Form}). Since the first appearance of the Gagliardo-Nirenberg interpolation inequality and the Sobolev inequality (see e.g. \cite{Brezisuni,Leoni}), authors of many papers and books consider a class of inequalities (maybe originated from some old literature, see \cite{Ilin,Ehrling,Bro1}), which may differ one from another because of the domain (and its regularity) of the function $u$ involved, and/or because of the parameters involved, and/or because of the use of modulars or norms, as it happened in the original results by Gagliardo and by Nirenberg (see the short historical background below).

In this paper we restrict our attention on the original results in \cite{Gag2,Nire1}, limiting ourselves to the case where the domain is the whole space. We do this in order to avoid the technicalities (for instance, working with the domains with the cone property as in E. Gagliardo's paper \cite{Gag2}).

The motivation of this paper is simple. The proof of Theorem \ref{MAINR}, perhaps because of its tedious steps, is missing except for the special cases of the new proofs of variants and generalizations and the new proofs which use techniques established after the original papers by E. Gagliardo and L. Nirenberg (see e.g. \cite{MazSha}). C. Miranda (see the paper \cite{Mirandalinc}, written in Italian), in order to motivate his improvement of the original result by E. Gagliardo and L. Nirenberg (see also the further extension by A. Canfora, \cite{Canfora}), considers the following inequality \eqref{mir}, special case ($k=2$, $j=1$, $\theta=1/2$), of inequality \eqref{GGNII}, except for the last term in the right hand side, due to the fact that \eqref{mir} has been stated in bounded domains with cone property: 
\begin{equation}\label{mir}
\|\nabla u\|_{p}\le C (\|\nabla^2u\|_{r}^{1/2}\|u\|_{q}^{1/2}+\|u\|_{q})\, , \qquad p=\frac{2rq}{r+q}\, .
\end{equation}
He writes that \eqref{mir} \sl actually does not have a full proof in literature. In fact, the proof by Nirenberg is given just through the main points \rm (by the way, this is written clearly by L. Nirenberg himself) \sl and just in the case $r>1$, while the E. Gagliardo's proof, even if complete in the case $r=1$, when $r>1$ is given when $q\ge r/(r-1)$. \rm In spite of the relevant refinement proven by C. Miranda in \cite{Mirandalinc}, unfortunately the proof of the Theorem \ref{MAINR} remains surprisingly missing. The reader, must go through some special cases, and some time must be spent to drop the technicalities to get any refinements. In the case of C. Miranda's paper \cite{Mirandalinc}, in his Theorem 1.1 he states a refinement for domains with the cone property, but the details of the proof are given just for functions of one variable. All the remaining machinery being made \sl as in Gagliardo, \rm while in his Theorem 4.1 he states a quite general result whose proof uses the original Gagliardo-Nirenberg inequality.

Our intention is to present the main part of the original proof by L. Nirenberg of the Theorem \ref{MAINR}, in an elementary, detailed version, accessible even for students with basic background in measure theory. It should be noticed that E. Gagliardo in \cite{Gag2} treats in fact just the case $\theta=j/k$ while L. Nirenberg in \cite{Nire1} treats the extreme cases $\theta=j/k$ and $\theta=1$. The latter one, which is the Sobolev inequality, is treated using an argument (already known since long time) by Morrey. The interpolation argument for general $\theta$ is mentioned in both papers by E, Gagliardo (\cite{Gag2}) and L. Nirenberg (\cite{Nire1}), but it is not developed in detail. The core of Theorem \ref{MAINR} is therefore in the case $\theta=j/k$ and it will be treated in this paper. The rough idea of the interpolation will be described after the proof of the case $\theta=j/k$. For details we refer to the recent paper \cite{MRS}. The result proven in Section \ref{prof}, after a short historical background given in Section \ref{histor}, is therefore the following special case of Theorem \ref{MAINR}:

\begin{Theorem}\label{MAINRproven}
If $1\le q\le \infty$, $1\le r< \infty$, $j,k\in \mathbb{N}$, $j<k$, and
\begin{equation}\label{GNGEproven}
\frac{1}{p}=\frac{j}{kr}+\frac{k-j}{kq}\, ,
\end{equation}
then there exists constant $C$ independent of $u$ such that
\begin{equation}\label{GGNIIproven}
\|\nabla^{j}u\|_{p}\leq C\|\nabla^{k}u\|_{r}^{j/k}\|u\|_{q}^{1-j/k}\qquad\forall u\in L^q(\mathbb{R}^n)\cap W^{k,r}(\mathbb{R}^n)\, .
\end{equation}
\end{Theorem}

\section{A short historical background}\label{histor}

The International Congress of Mathematicians was held in Edinburgh on 14-21 August, 1958, attended by 1658 full members and 757 associate members, the largest total number for any International Congress of Mathematicians until that year. Besides the 19 one-hour lectures and 37 half-hour lectures, among the 604 fifteen-minutes short communications there were that one by E. Gagliardo (in French language), entitled \sl Propri\'et\'es de certaines classes de fonctions de $n$ variables, \rm and that one by L. Nirenberg, entitled \sl Inequalities for derivatives. \rm Both speakers went to Edinburgh to present to the mathematical community their results, before presenting them to some journals. Both of them, while discussing in Edinburgh before their respective communications, discovered with great surprise to present extremely similar estimates for intermediate derivatives. This was unfortunately affecting the novelty of their announcements. In a friendly atmosphere, both of them decided to mention in their communications the coincidence of the results and the discovery made in the Congress. The respective citations appeared also in their papers. E. Gagliardo just added an appendix in his paper \cite{Gag2} (and made the submission on November 8, 1958), where he put one of his inequalities (see \eqref{ga1} below) in the (\sl more elegant, \rm E. Gagliardo writes) form \eqref{GGNII} stated by L. Nirenberg (to be precise, from inequality \eqref{ga1} he got \eqref{GGNII} in the case $\theta=j/k$). Even tough E. Gagliardo mentioned the interpolation argument to get the result for all the range of $\theta$'s, it must be noticed that the full range of parameters in \eqref{paramet} does not appear, while L. Nirenberg inserted the inequality in one of his Conferences, given in the framework of a course C.I.M.E. (Centro Internazionale Matematico Estivo) held in Pisa from 1 to 10 September 1958, subsequently published (see \cite{Nire1}) in 1959.

The original Gagliardo's inequality (Theorem 7.I in \cite{Gag2}) is not a norm inequality, but a modular-type inequality. In our notation (here and in the following we change the notation for the indices according to those considered in the statement of Theorem \ref{MAINR}), Gagliardo, for $u\in L^q(\Omega)\cap W^{k,r}(\Omega)$, $\Omega\subset \mathbb{R}^n$ bounded open set with the cone property, proved that
\begin{equation}\label{ga1}
\|\nabla^{j}u\|_{p}^p\leq C\left(\|\nabla^{k}u\|_{r}^{r}+\|u\|_{q}^q+1\right)\, \quad 0<j<k\, , 1\le r\le q\, , p=\frac{kqr}{jq+(k-j)r}\, , r\ge 2
\end{equation}
considering first (see Theorem 6.I in \cite{Gag2}) the case $k=2$ and two possibilities for $r$: $r=1$ or $r>1$, the latter case with the restriction $q\ge r/(r-1)$. The value of $p$, here, is the same as in \eqref{GNGEproven} setting $k=2$, $j=1$.

The proof by E. Gagliardo for general $r$ is reduced in few lines to the case $r=2$ and, similarly as we will do in this paper, the technicalities due to the assumption that $\Omega$ has the cone property were avoided in few lines. We, however, decided to state the result directly for $\Omega=\mathbb{R}^n$. In his paper, E. Gagliardo wrote that \sl by approximation one can make the proof just for $C^\infty$ functions \rm and that \sl because of a property proved in \cite{Gag1} \rm (a paper by himself, published the year before), \sl it suffices to make the proof for domains $\Omega$ which are a union of cubes with side-length $l$, with sides parallel to the coordinate axes, such that the intersection of $\Omega$ with straight lines parallel to the coordinate axes is exactly one segment having length between $l$ and $2l$. \rm Again in few lines (this reduction is now reasonable, however, the easy integrations over lines are among the few steps described in L. Nirenberg's paper \cite{Nire1}), the proof was reduced to the following one dimensional inequality:
\begin{equation}\label{gaonedim}
\int_a^{a+\lambda}|u'|^{p}dx\le c\left(\lambda^{-p}\int_a^{a+\lambda}|u|^{q}dx+
\lambda^{2r-p}\int_a^{a+\lambda}|u''|^{r}dx+\lambda^{1-p}
\right)
\end{equation} whose proof was divided into two cases, $r>1$ and $r=1$. Even if the paper by E. Gagliardo treats formally only the case of bounded domains with the cone property (even in the comparison with the L. Nirenberg's result), it must be noticed that from inequality \eqref{ga1} one can get the statement of Theorem \ref{MAINR}. L. Nirenberg wrote in \cite{Nire2} (see the end of the first section of the paper, p. 734), the statement for the whole space $\mathbb{R}^n$ comes from the analogous one for a cube.

We conclude this section by telling that the rigorous details behind the technicalities hidden in arguments involving domains with the cone property are frequently missing. It is worth to mention, in this order of ideas, the book \cite{RFio} by R. Fiorenza (the authors thank him for discussions on the historical background above, being himself in the same Congress, speaker of a short communication), which is a strong revision of certain original notes by C. Miranda contained in \cite{MirIta}.

\section{The proof}\label{prof}

Let $M\subset \mathbb{R}^n$ be a set of positive, finite Lebesgue measure. Let $1\le p<\infty$ and let $u$ be Lebesgue measurable function defined on $M$. We shall use the symbol $\|\cdot\|_{p,M}$ to denote the norm in the space $L^{p}(M)$, defined through
$$
\|u\|_{p,M}=\left(\int_{M}|u|^{p}\right)^{\frac{1}{p}}.
$$
If the set $M$ is omitted in the symbol of the norm, we mean that the underlying set is the whole space ($\mathbb{R}$ or $\mathbb{R}^n$, no confusion will arise). The norm in the case $p=\infty$ means the essential supremum over $M$.

\subsection{Auxiliary results}

We begin with the $L^p$ version of a simple, known, observation (see e.g. \cite{Tor}, Prop. 1.1 p. 200), which in fact holds true in any Banach function space.
\begin{Lemma}\label{IMVL}
If $1\le p\le \infty$ and $M\subset \mathbb{R}^n$ is a set of positive, finite Lebesgue measure, then
$$
\|u-u_{M}\|_{p,M}\leq 2\displaystyle{\inf_{c\in\mathbb{R}}}\|u-c\|_{p,M}.
$$
where
$$
u_{M}:=\frac{1}{|M|}\int_{M}u(x)\textup{d}x;
$$
\end{Lemma}
\begin{proof}
By H\"older inequality we get
$$
\begin{aligned}
\|u-u_{M}\|_{p,M}&\leq\|u-c\|_{p,M}+\|u_{M}-c\|_{p,M}\\
&=\|u-c\|_{p,M}+\left(\int_M\left(\frac{1}{|M|}\int_M u(x)\textup{d}x-\frac{1}{|M|}\int_M c\,\textup{d}x\right)^p\right)^{\frac{1}{p}}\\
&\leq\|u-c\|_{p,M}+\frac{1}{|M|}\|u-c\|_{1,M}|M|^{1/p}\\
&\leq\|u-c\|_{p,M}+\frac{1}{|M|^{1/p}}\|u-c\|_{p,M}|M|^{1/p}=2\|u-c\|_{p,M}\, ,
\end{aligned}
$$
and the proof is completed taking the infimum over all $c$.
\end{proof}

\begin{Lemma}\label{Klic}
Let $p,q,r\geq 1$,  let $I\subset\mathbb{R}$ be an open interval of length $l$. Then the inequality
$$
\| u'\|_{p,I}\leq C\left(l^{1+\frac{1}{p}-\frac{1}{r}}\|u''\|_{r,I}+l^{-1+\frac{1}{p}-\frac{1}{q}}\|u\|_{q,I}\right)
$$
holds for any function $u\in W^{2,1}(I)$.
\end{Lemma}
\begin{proof}
By the previous lemma we have
$$
\|u-u_{I}\|_{q,I}\approx\displaystyle{\inf_{c\in\mathbb{R}}}\|u-c\|_{q,I},
$$
hence, we may assume that
$$
\int_{I}u(x)\textup{d}x=0.
$$
Denote by
$$
d:=\frac{1}{l}\int_{I}u'(x)\textup{d}x.
$$
Now we denote by $x_{0}$ the center of $I$ and we set
\begin{equation}\label{utilde}
\tilde{u}(x):=u(x)-d\left(x-x_{0}\right)\, , \quad x\in I\, ,
\end{equation}
so that
\begin{equation}\label{meanuprime}
\int_{I}\tilde{u}(x)\textup{d}x=
\int_{I} \tilde{u}'(x)\textup{d}x=0.
\end{equation}
Since, in general,
\begin{equation}\label{MOR}
|v(x)-v_{I}|\leq \|v'\|_{r,I}|I|^{\frac{r-1}{r}}\qquad \forall
v\in W^{1,r}(I)\, ,\,\, r\ge 1\, ,
\end{equation}
from \eqref{meanuprime} we have $\tilde{u}'_I=0$ and therefore
\begin{equation}\label{uprim}
\begin{aligned}
\|\tilde{u}'\|_{p,I}=\|\tilde{u}'-\tilde{u}'_I\|_{p,I}
\leq \|\|\tilde{u}''\|_{r,I}|I|^{\frac{r-1}{r}}\|_{p,I}=l^{\frac{r-1}{r}+\frac{1}{p}}\|u''\|_{r,I}\, ,
\end{aligned}
\end{equation}
and therefore, using again \eqref{meanuprime} and the estimate in \eqref{MOR}, we obtain
\begin{equation}\label{secondest}
\|\tilde{u}\|_{q,I}=\|\tilde{u}-\tilde{u}_I\|_{q,I}
\leq \|\|\tilde{u}'\|_{p,I}|I|^{\frac{p-1}{p}}\|_{q,I}
\leq \|l^{\frac{r-1}{r}+\frac{1}{p}}\|u''\|_{r,I}l^{\frac{p-1}{p}}\|_{q,I}
= l^{1+\frac{1}{q}+\frac{r-1}{r}}\|u''\|_{r,I}\, .
\end{equation}
Now, using in turn \eqref{utilde}, \eqref{uprim}, \eqref{utilde} again, and finally \eqref{secondest},
$$
\begin{aligned}
\|u'\|_{p,I}&\leq \|\tilde{u}'\|_{p,I}+\|d\|_{p,I}\\
&=\|\tilde{u}'\|_{p,I}+\frac{\|1\|_{p,I}}{\|x-x_{0}\|_{q,I}}\|d(x-x_{0})\|_{q,I}\\
&\lesssim l^{1+\frac{1}{p}-\frac{1}{r}}\|u''\|_{r,I}+l^{-1+\frac{1}{p}-\frac{1}{q}}\|d(x-x_{0})\|_{q,I}\\
&\lesssim l^{1+\frac{1}{p}-\frac{1}{r}}\|u''\|_{r,I}+l^{-1+\frac{1}{p}-\frac{1}{q}}(\|u\|_{q,I}+\|\tilde{u}\|_{q,I})\\
&\lesssim l^{1+\frac{1}{p}-\frac{1}{r}}\|u''\|_{r,I}+l^{-1+\frac{1}{p}-\frac{1}{q}}\|u\|_{q,I},
\end{aligned}
$$
which is the estimate from the statement.
\end{proof}

It is our goal to prove Theorem \ref{MAINRproven} in the assumption \eqref{GNGEproven}, however,
in next three lemmata, we will examine a special case and we will assume, even if not stated explicitly, that
$1\leq q\leq \infty$, $1\leq r< \infty$, and $1\leq p< \infty$ is such that
\begin{equation}\label{speciallink}
\frac{2}{p}=\frac{1}{r}+\frac{1}{q}\, .
\end{equation}

Even if it is not of interest for our goals, we notice that next Lemma \ref{primo} holds also for a wider range of exponents, not necessarily linked by \eqref{speciallink}.

\begin{Lemma}\label{primo}
Let $u\in\mathcal{C}^{\infty}_{c}(\mathbb{R})$. There exists a sequence of open intervals $(I_{k})$, which covers the (compact) support of $u$, such that
\begin{enumerate}[\upshape(i)]
\item $
l_{k}^{1+\frac{1}{p}-\frac{1}{r}}\|u''\|_{r,I_{k}}=l_{k}^{-1+\frac{1}{p}-\frac{1}{q}}\|u\|_{q,I_{k}};
$
\vskip.4cm\item
$
\qquad\qquad\qquad\qquad\qquad\qquad\qquad\displaystyle{\sum_{k}\chi_{I_{k}}}\leq 4,
$
\end{enumerate}

\noindent
where $l_{k}$ stands for length of $I_{k}$.
\end{Lemma}
\begin{proof}
Let us consider a non-zero function $u$. Without loss of generality we may suppose that $u(0)\neq 0$. Given a point $x$ in the support of $u$, define functions $\alpha_{x}$ and $\omega_{x}$ by
$$
\omega_{x}(h):=h^{1+\frac{1}{p}-\frac{1}{r}}\|u''\|_{r,(x-\frac{h}{2},x+\frac{h}{2})}
$$
and
$$
\alpha_{x}(h):=h^{-1+\frac{1}{p}-\frac{1}{q}}\|u\|_{q,(x-\frac{h}{2},x+\frac{h}{2})}.
$$
Note that both functions are continuous and $\omega_{x}(+\infty)=\alpha_{x}(0+)=\infty$, while\\ $\omega_{x}(0+)~=~\alpha_{x}(+\infty)=0$. Set
$$
r_{x}:=\sup \{h>0: \omega_{x}(h)\leq\alpha_{x}(h)\}.
$$
Denoting
$$
M:=2\max\{r_{0},\textup{diam}(\textup{supt}(u))\}\, ,
$$
for any $x\in\textup{supt}(u)$ one has $r_{x}\leq M$: in fact, by contradiction, assume temporarily
that $r_{x}> M$. We would have the existence of $h_M>M$ such that $\omega_{x}(h_M)\leq\alpha_{x}(h_M)$. But, since $h_M>M\ge 2\textup{diam}(\textup{supt}(u))$, we would get $\omega_{x}(h_M)=\omega_{0}(h_M)$, $\alpha_{x}(h_M)=\alpha_{0}(h_M)$. This implies $\omega_{0}(h_M)\leq\alpha_{0}(h_M)$, hence $h_M\le r_0<M$, against $h_M>M$, which is a contradiction.

Let us consider covering by open intervals
$$
\{I_{x}\}_{x\in \textup{supt}(u)}=(x-r_{x},x+r_{x}).
$$
Now, since $r_{x}$ are uniformly bounded, we may choose subcovering with the property (ii) (one can use Besicovitch theorem, or prove this directly in dimension one, which is a simple exercise).
\end{proof}

The following lemma should be compared with inequality (2.6) in \cite{Nire1}, proved by L.~Nirenberg using a covering argument which in fact has been formalized in the previous Lemma \ref{primo}.

\begin{Lemma}\label{K}
Let $u\in\mathcal{C}^{2}(\mathbb{R})\cap L^{q}(\mathbb{R})$ be such that $u''\in L^{r}(\mathbb{R})$. Then $u'\in L^{p}(\mathbb{R})$ and there exists a constant $C$ independent of $u$ such that
\begin{equation}\label{unodimens}
\|u'\|_{p}^{2}\leq C \|u''\|_{r}\|u\|_{q}.
\end{equation}
\end{Lemma}
\begin{proof}
First suppose that $u\in \mathcal{C}^{\infty}_{c}$. Let $(I_{k})$ be a family of intervals satisfying properties (i) and (ii) from the previous lemma. Then the following estimates hold
$$
\begin{aligned}
\|u'\|_{p}^{p}&\leq \displaystyle{\sum_{k}} \int_{I_{k}}|u'(s)|^{p}\textup{d}s\\
&\lesssim \displaystyle{\sum_{k}}l_{k}^{p+1-\frac{p}{r}}\|u''\|_{r,I_{k}}^{p}\\
&= \displaystyle{\sum_{k}}\|u''\|_{r,I_{k}}^{\frac{p}{2}}\cdot l_{k}^{p+1-\frac{p}{r}}\|u''\|_{r,I_{k}}^{\frac{p}{2}}\\
&= \displaystyle{\sum_{k}}\|u''\|_{r,I_{k}}^{\frac{p}{2}}\cdot l_{k}^{p+1-\frac{p}{r}}(l_{k}^{-2-\frac{1}{q}+\frac{1}{r}}\|u\|_{q,I_{k}})^{\frac{p}{2}}\\
&=\displaystyle{\sum_{k}}\left(\|u''\|_{r,I_{k}}^{r}\right)^{\frac{q}{q+r}}\left(\|u\|_{q,I_{k}}^{q}\right)^{\frac{r}{q+r}}l_{k}^{p+1-\frac{p}{r}+\frac{p}{2}\left(-2-\frac{1}{q}+\frac{1}{r}\right)}\\
&=\displaystyle{\sum_{k}}\left(\|u''\|_{r,I_{k}}^{r}\right)^{\frac{q}{q+r}}\left(\|u\|_{q,I_{k}}^{q}\right)^{\frac{r}{q+r}}\\
&\leq \left(\displaystyle{\sum_{k}}\|u''\|_{r,I_{k}}^{r}\right)^{\frac{q}{q+r}}\left(\displaystyle{\sum_{k}}\|u\|_{q,I_{k}}^{q}\right)^{\frac{r}{r+q}}\\
&\lesssim \|u''\|_{r}^{\frac{p}{2}}\|u\|_{q}^{\frac{p}{2}}\, .
\end{aligned}
$$
To complete the proof define
$$
u_{\varepsilon}:=\varphi_{\varepsilon}*u
$$
where $\varphi$ is standard mollifier, and take the limit where $\varepsilon\rightarrow 0+$.
\end{proof}

\begin{Lemma}\label{GN12}
If $u\in W^{2,r}(\mathbb{R}^n)$, $1\leq p,q,r$ and
$$
\frac{2}{p}=\frac{1}{q}+\frac{1}{r},
$$
then
$
\|\nabla u\|_{p}^{2}\lesssim\|\nabla^{2} u\|_{r}\|u\|_{q}
$.
\end{Lemma}
\begin{proof}
By Lemma \ref{K} we have that the statement holds in the case of $n=1$.

Before giving the second step of the induction argument, let us fix some notation we will use in this proof. For a $n$-dimensional vector $x=(x_{1},x_{2},\dots x_{n})$ we set $x'=(x_{1},x_{2},\dots,x_{n-1})$ and therefore $x=(x_{1},x_{2},\dots,x_{n-1},x_{n})=(x',x_{n})$, so that it will be natural to use the symbols
$$
u_{x'}:=\left(\frac{\partial u}{\partial x_{1}},\frac{\partial u}{\partial x_{2}},\dots \frac{\partial u}{\partial x_{n-1}}\right)
$$
and
$$
u_{x_{i}}=\frac{\partial u}{\partial x_{i}}\, .
$$
Moreover, for $\tilde{x}=(x_{i})_{i\in A}$,  $A=\{i_{1},i_{2},\dots i_{k}\}\subset \{1,2,\dots,n\}$, we denote
$$
\|u\|_{p,\tilde{x}}:=\left(\int_{\mathbb{R}}\int_{\mathbb{R}}\dots\int_{\mathbb{R}}|u(x)|^{p}\textup{d}x_{i_{1}}\dots\textup{d}x_{i_{k-1}}\textup{d}x_{i_k}\right)^{\frac{1}{p}}\, .
$$
Note that $\|u\|_{p, x'}$ is a function of $x_{n}$ and $\|u\|_{p, x_{n}}$ is function of $x'$.

Now suppose the theorem holds for dimension $n-1$.
We use the Fubini theorem, the H\"older's inequality and the inequalities $|u_{x'x'}|=|(u_{x'})_{x'}|\leq |\nabla^{2} u|$ and $|u_{x_{n}x_{n}}|\leq|\nabla^{2} u|$ which
hold pointwise, to get
$$
\begin{aligned}
\|\nabla u\|_{p}^{p}\approx& \int_{\mathbb{R}}\int_{\mathbb{R}^{n-1}}|u_{x'}|^{p}\textup{d}x'\textup{d}x_{n}+\int_{\mathbb{R}^{n-1}}\int_{\mathbb{R}}|u_{x_{n}}|^{p}\textup{d}x_{n}\textup{d}x'\\
\lesssim& \int_{\mathbb{R}}\|u_{x'x'}\|^{\frac{p}{2}}_{r,x'}\|u\|_{q,x'}^{\frac{p}{2}}\textup{d}x_{n}+\int_{\mathbb{R}^{n-1}}\|u_{x_{n}x_{n}}\|_{r,x_{n}}^{\frac{p}{2}}\|u\|_{q,x_{n}}^{\frac{p}{2}}\textup{d}x'\\
\leq& \left(\int_{\mathbb{R}}\|u_{x'x'}\|_{r,x'}^{r}\textup{d}x_{n}\right)^{\frac{p}{2r}}\left(\int_{\mathbb{R}}\|u\|_{q,x'}^{q}\textup{d}x_{n}\right)^{\frac{p}{2q}}\\
&+\left(\int_{\mathbb{R}^{n-1}}\|u_{x_{n}x_{n}}\|_{r,x_{n}}^{r}\textup{d}x'\right)^{\frac{p}{2r}}\left(\int_{\mathbb{R}^{n-1}}\|u\|_{q,x_{n}}^{q}\textup{d}x'\right)^{\frac{p}{2q}}\\
\lesssim& \|\nabla^{2}u\|_{r}^{\frac{p}{2}}\|u\|_{q}^{\frac{p}{2}}.
\end{aligned}
$$

\end{proof}

\subsection{Proof of Theorem \ref{MAINRproven}}

We shall use notation $GN(k,j)$ to express that Theorem \ref{MAINRproven} holds for upper derivative of order $k$ and lower derivative of order $j$. By the previous lemma we have $GN(2,1)$. We shall prove two induction steps.

First let us prove that $GN(k,1)$ implies $GN(k+1,1)$. Let $p,q,r$ be numbers for which
\begin{equation}\label{firstelem}
\frac{1}{p}=\frac{\frac{1}{k+1}}{r}+\frac{\frac{k}{k+1}}{q}
\end{equation}
and let $\tilde{p}$ be a number defined by
\begin{equation}\label{secondelem}
\frac{2}{p}=\frac{1}{\tilde{p}}+\frac{1}{q}.
\end{equation}
By $GN(2,1)$ we obtain
$$
\|\nabla u\|_{p}\leq C\|\nabla^{2} u\|_{\tilde{p}}^{\frac{1}{2}}\|u\|_{q}^{\frac{1}{2}}
$$
Now we apply $GN(k,1)$ on $\nabla u$ (with $q,p$ replaced by $p,\tilde{p}$ respectively, this is allowed by relations \eqref{firstelem} and  \eqref{secondelem}) and get
$$
\|\nabla^{2} u\|_{\tilde{p}}\leq C\|\nabla^{k+1} u\|_{r}^{\frac{1}{k}}\|\nabla u\|_{p}^{1-\frac{1}{k}}.
$$
Combination of last two estimates yields
$$
\|\nabla u\|_{p}\leq C\|\nabla^{k+1} u\|_{r}^{\frac{1}{k+1}}\|u\|_{q}^{\frac{k}{k+1}},
$$
which is $GN(k+1,1)$.

In the second induction step, we shall prove that $GN(k,j)$ implies $GN(k+1,j+1)$. Let $p,q,r$ be numbers satisfying
$$
\frac{1}{p}=\frac{\frac{j+1}{k+1}}{r}+\frac{\frac{k-j}{k+1}}{q}
$$
and let $\tilde{q}$ be number for which
$$
\frac{1}{p}=\frac{\frac{j}{k}}{r}+\frac{1-\frac{j}{k}}{\tilde{q}}.
$$
By $GN(k,j)$ applied on function $\nabla u$ we have
\begin{equation}\label{FRS}
\|\nabla^{j+1}u\|_{p}\leq C\|\nabla^{k+1} u\|_{r}^{\frac{j}{k}}\|\nabla u\|_{\tilde{q}}^{1-\frac{j}{k}}.
\end{equation}
Since by the previous step we may assume that $GN(j+1,1)$ holds, we have
\begin{equation}\label{SC}
\|\nabla u\|_{\tilde{q}}\leq C\|\nabla^{j+1} u\|_{p}^{\frac{1}{j+1}}\|u\|_{q}^{\frac{j}{j+1}}.
\end{equation}
Combining the estimates \eqref{FRS} and \eqref{SC} we finish the proof.

\subsection{Sharpness of \eqref{GNGE}}

We conclude showing that the relationship between parameters in \eqref{GNGE} is essential (by \sl dimensional analysis, \rm as asserted in \cite{Nire1}) for Theorem \ref{MAINRproven} to hold. Consider a non-zero function $u\in\mathcal{C}_{0}^{\infty}(\mathbb{R}^{n})$ and define a dilation operator by
$$
T_{s}u(x):=u(sx)\, .
$$
Assume that there exists a constant $C$ independent of $v$ such that
\begin{equation}\label{optimalityy}
\|\nabla^{j}v\|_{p}\leq C\|\nabla^{k}v\|_{r}^{\theta}\|v\|_{q}^{1-\theta}
\end{equation}
Setting $v=T_{s}u$, for all $s\in(0,\infty)$ we have
$$
\nabla^{j}v=s^{j}T_{s}(\nabla^{j} u), \quad \nabla^{k}v=s^{k}T_{s}(\nabla^{k} u)\, ;
$$
on the other hand, since for any $f\in L^{m}(\mathbb{R}^n)$,  $m\ge 1$, one has
$$
\|T_{s}f\|_{m}=\left(\int_{\mathbb{R}^{n}}|f(sx)|^{m}\textup{d}x\right)^{\frac{1}{m}}=s^{-\frac{n}{m}}\|f\|_{m},
$$
and inequality \eqref{optimalityy} holds if and only if
$$
s^{j-\frac{n}{p}}\|\nabla^{j} u\|_p\leq C s^{\theta(k-\frac{n}{r})}\|\nabla^{k} u\|_{r}^{\theta}s^{-(1-\theta)\frac{n}{q}}\|u\|_{q}^{1-\theta}.
$$
Such inequality can be valid for all positive $s$ only if
$$
j-\frac{n}{p}=\theta\left(k-\frac{n}{r}\right)-(1-\theta)\frac{n}{q},
$$
which is equivalent to \eqref{GNGE}.
\begin{Remark}
Let us briefly explain how the interpolation works. If $p,r\geq 1$ are given, we first use the estimate
\begin{equation}\label{INTERP}
\|u\|_{p}\leq \|u\|^{\alpha}_{r}\|u\|^{1-\alpha}_{q},
\end{equation}
where
$$
\frac{1}{p}=\frac{\alpha}{r}+\frac{1-\alpha}{q}.
$$
Now if the parameter $\alpha$ is chosen properly we obtain the product of powers of two norms: one to be estimated by the use of the Sobolev embedding theorem, and the other by Theorem \ref{MAINRproven}. The L. Nirenberg's version of the theorem admitted $p,r<0$, the appropriate $p,q$-norms being defined as H\" older spaces (for precise definition see original paper). L. Nirenberg claimed that the inequality \eqref{INTERP} holds even in this extended environment. However the proof is missing, neither a reference is given. A partial result is given by A. Kufner and A. Wannebo in \cite{Kuf}, and a proof covering all the cases is contained in a recent paper, see \cite{MRS}.
\end{Remark}
\subsection{Comparison with Nirenberg's proof}

Lemma\ref{Klic} was stated in the original paper by L. Nirenberg in the same form (see (2.7) p. 130 in \cite{Nire1}), it's proof, however was given as an exercise. The careful reader may notice different exponents in the second term of (2.7) in \cite{Nire1} and in the analogous exponent in next Lemma: it should be stressed that the statement by Nirenberg is given only for $\theta=1/2$. If in the Lemma below one replaces $q$ in terms of $r$ and $p$ in the case $\theta=1/2$ (see see (2.5) p. 129 in \cite{Nire1}), then the exponents coincide.

\begin{Remark}
Lemma \ref{Klic} is the first, key step to get Theorem \ref{MAINRproven}. However,
it may be interesting to see also how from the inequality \eqref{gaonedim} by E. Gagliardo, which can be equivalently written, in the notation of the Lemma, as follows:
\begin{equation}\label{gaonedimnew}
\| u'\|_{p,I}\leq C\left(l^{2\frac{r}{p}-1}\|u''\|_{r,I}^{\frac{r}{p}}+l^{-1}\|u\|_{q,I}^{\frac{q}{p}}+l^{\frac{1}{p}-1}\right)\, ,
\end{equation}
one can get the inequality \eqref{GGNIIproven} (that one in the statement of Theorem \ref{MAINRproven}) for functions supported in a bounded interval, in the one-dimensional case, when $j=1$, $k=2$ (compare with \eqref{unodimens} below, stated for functions defined in all $\mathbb{R}$), hence when the relation among exponents is
\begin{equation}\label{gacase}
\frac{1}{p}=\frac{1}{2r}+\frac{1}{2q}\, .
\end{equation}
We notice that E. Gagliardo in his appendix compared his inequality with that one by L. Nirenberg only in the case of bounded open sets, hence he deduced an inequality with an extra term on the right hand side with respect to \eqref{GGNIIproven}.

The first remark is that the third term in the right hand side of \eqref{gaonedimnew} can be dropped by a standard homogenizing argument (for a formalization in a quite general context see \cite{daristfio}): one applies the inequality to $u$ replaced, say, by $tu$ and then, using that the constant in the right hand side does not depend on $u$ (hence it does not depend on $t$), letting $t\to\infty$, one gets
\begin{equation}\label{gaonedimnewbis}
\| u'\|_{p,I}\leq C\left(l^{2\frac{r}{p}-1}\|u''\|_{r,I}^{\frac{r}{p}}+l^{-1}\|u\|_{q,I}^{\frac{q}{p}}\right)\, ,
\end{equation}
Writing $l=\alpha \|u''\|_{r,I}^\lambda \|u\|_{q,I}^\mu$ for some $\lambda,\mu$ to be determined later, let us consider the interesting case $\alpha\|u''\|_{r,I}>0$. The first term in the right hand side of \eqref{gaonedimnewbis} becomes
\begin{equation}\label{firstterm}
l^{2\frac{r}{p}-1}\|u''\|_{r,I}^{\frac{r}{p}}=\alpha^{2\frac{r}{p}-1}
\|u''\|_{r,I}^{\left(2\frac{r}{p}-1\right)\lambda+\frac{r}{p}}\|u\|_{q,I}^{\left(2\frac{r}{p}-1\right)\mu}\, ,
\end{equation}
while the second term becomes
\begin{equation}\label{secondterm}
l^{-1}\|u\|_{q,I}^{\frac{q}{p}}=\alpha^{-1}
\|u''\|_{r,I}^{-\lambda}\|u\|_{q,I}^{-\mu+\frac{q}{p}}\, .
\end{equation}
If we look for $\lambda,\mu$ such that exponents of $\|u''\|_{r,I}$, $\|u\|_{q,I}$, coincide, we get easily $\lambda=-1/2$, $\mu=q/2r$, which, inserted into the right hand sides of \eqref{firstterm} and/or \eqref{secondterm}, give terms of the type $\|u''\|_{r,I}^{\frac{1}{2}}\|u\|_{q,I}^{\frac{q}{p}-\frac{q}{2r}}$. Taking into account that E. Gagliardo got the inequality in the case \eqref{gacase}, both exponents are exactly $1/2$, i.e. we get the right hand side term of \eqref{GGNIIproven} when $j=1$, $k=2$, as asserted.
\end{Remark}

The third and last lemma should be compared with inequality (2.5) in \cite{Nire1}, proven by L. Nirenberg simply telling that it follows from \eqref{unodimens} \sl by integrating with respect to the other variables and applying H\"older's inequality. \rm

In his proof L. Nirenberg completely omits the step showing the Theorem\ref{MAINRproven} claiming that \textit{the proof is done by induction}.
\bigskip

\noindent
\it Acknowledgments. \rm The second author has been partially supported by the Gruppo Nazionale per l'Analisi Matematica, la Probabilit\`a e le loro Applicazioni (GNAMPA) of the Istituto Nazionale di Alta Matematica (INdAM) and by Universit\`a di Napoli Parthenope through the project ``Sostegno alla Ricerca individuale'' (triennio 2015-2017),
the third author was supported by GA\v{C}R 18--00960Y, and the fourth author was supported by EF-IGS2017-Soudsk\' y-IGS07P1.


\begin{thebibliography}{10}

\bibitem{AdaFou}
R.~A. Adams and J.~J.~F. Fournier.
\newblock {\em Sobolev spaces}, volume 140 of {\em Pure and Applied Mathematics
  (Amsterdam)}.
\newblock Elsevier/Academic Press, Amsterdam, second edition, 2003.

\bibitem{Brezisuni}
H.~Brezis.
\newblock {\em Functional analysis, {S}obolev spaces and partial differential
  equations}.
\newblock Universitext. Springer, New York, 2011.

\bibitem{BrezisGaNi}
H.~Brezis and P.~Mironescu.
\newblock Gagliardo-{N}irenberg, composition and products in fractional
  {S}obolev spaces.
\newblock {\em J. Evol. Equ.}, 1(4):387--404, 2001.
\newblock Dedicated to the memory of Tosio Kato.

\bibitem{Bro1}
F.~E. Browder.
\newblock On the eigenfunctions and eigenvalues of the general linear elliptic
  differential operator.
\newblock {\em Proc. Nat. Acad. Sci. U. S. A.}, 39:433--439, 1953.

\bibitem{Bro2}
F.~E. Browder.
\newblock On the spectral theory of elliptic differential operators. {I}.
\newblock {\em Math. Ann.}, 142:22--130, 1960/1961.

\bibitem{Canfora}
A.~Canfora.
\newblock Extension of an integral inequality of {C}. {M}iranda and
  applications to the elliptic equations with discontinuous coefficients.
\newblock {\em Czechoslovak Math. J.}, 39(114)(3), 1989.

\bibitem{CapFioKala}
C.~Capone, A. Fiorenza and A.~Ka\l amajska.
\newblock Strongly nonlinear {G}agliardo-{N}irenberg inequality in {O}rlicz
  spaces and {B}oyd indices.
\newblock {\em Atti Accad. Naz. Lincei Rend. Lincei Mat. Appl.},
  28(1):119--141, 2017.

\bibitem{CriMare}
F.~Crispo and P.~Maremonti.
\newblock An interpolation inequality in exterior domains.
\newblock {\em Rend. Sem. Mat. Univ. Padova}, 112:11--39, 2004.

\bibitem{daristfio}
A.~M.~D'Aristotile and A.~Fiorenza.
\newblock A topology on inequalities.
\newblock {\em Electron. J. Differential Equations}, 85:1--22, 2006.

\bibitem{Dibene}
E.~DiBenedetto.
\newblock {\em Real analysis}.
\newblock Birkh\"auser Advanced Texts: Basler Lehrb\"ucher. [Birkh\"auser
              Advanced Texts: Basel Textbooks], 2016

\bibitem{Ehrling}
G.~Ehrling.
\newblock On a type of eigenvalue problems for certain elliptic differential
  operators.
\newblock {\em Math. Scand.}, 2:267--285, 1954.

\bibitem{EvaGar}
L.~C. Evans and R.~F. Gariepy.
\newblock {\em Measure theory and fine properties of functions}.
\newblock Textbooks in Mathematics. CRC Press, Boca Raton, FL, revised edition,
  2015.

\bibitem{RFio}
R.~Fiorenza.
\newblock {\em H\"older and locally {H}\"older continuous functions, and open
  sets of class {$C^k, C^{k,\lambda}$}}.
\newblock Frontiers in Mathematics. Birkh\"auser/Springer, Cham, 2016.

\bibitem{Form}
M.~R.~Formica.
\newblock A multiplicative embedding inequality in Orlicz-Sobolev spaces.
\newblock {\em J. Inequal. Pure Appl. Math.}, 1: art. 33, 6 pp., 7(2006).

\bibitem{Frie}
A.~Friedman.
\newblock {\em Partial differential equations. Corrected reprint of the original edition}.
\newblock Robert E. Krieger Publishing Co., Huntington, N.Y., 1976.

\bibitem{Gag1}
E.~Gagliardo.
\newblock Propriet\`a di alcune classi di funzioni in pi\`u variabili.
\newblock {\em Ricerche Mat.}, 7:102--137, 1958.

\bibitem{Gag2}
E.~Gagliardo.
\newblock Ulteriori propriet\`a di alcune classi di funzioni in pi\`u
  variabili.
\newblock {\em Ricerche Mat.}, 8:24--51, 1959.

\bibitem{Giannetti}
F.~Giannetti.
\newblock The modular interpolation inequality in {S}obolev spaces with
  variable exponent attaining the value 1.
\newblock {\em Math. Inequal. Appl.}, 14(3):509--522, 2011.

\bibitem{GilTru}
D.~Gilbarg and N.~S. Trudinger.
\newblock {\em Elliptic partial differential equations of second order}.
\newblock Classics in Mathematics. Springer-Verlag, Berlin, 2001.
\newblock Reprint of the 1998 edition.

\bibitem{Golo}
K.~K.~Golovkin.
\newblock On imbedding theorems.
\newblock {\em Soviet Math. Dokl.}, 1:998--1000, 1960.

\bibitem{Ilin}
V.~P.~Il'in.
\newblock Some inequalities in function spaces and their application to the
  investigation of the convergence of variational processes.
\newblock {\em The works on the approximate analysis, Trudy Mat. Inst.
  Steklov.}, 53:64--127, 1959.
 
\bibitem{KalaKrb}
A.~Ka\l amajska and M.~Krbec.
\newblock Gagliardo-{N}irenberg inequalities in regular {O}rlicz spaces
  involving nonlinear expressions.
\newblock {\em J. Math. Anal. Appl.}, 362(2):460--470, 2010.

\bibitem{Kjf}
A.~Kufner, O. John and S.~Fu\v c\'\i k.
\newblock {\em Function spaces}.
\newblock Noordhoff International Publishing, Leyden; Academia, Prague, 1977.
\newblock Monographs and Textbooks on Mechanics of Solids and Fluids;
  Mechanics: Analysis.
  
\bibitem{Kuf}
A.~Kufner and A.~Wannebo.
\newblock {An Interpolation Inequality Involving
H\" older Norms}.
\newblock {\em Georg. Math. J.}, 2:603--612, 1995.

\bibitem{Lady}
O.~A.~Ladyzhenskaya and N.~N.~Ural'tseva.
\newblock {\em Linear and quasilinear elliptic equations}.
\newblock Translated from the Russian by Scripta Technica, Inc. Translation
  editor: Leon Ehrenpreis. Academic Press, New York-London, 1968.

\bibitem{Leoni}
G.~Leoni.
\newblock {\em A first course in {S}obolev spaces}, volume 181 of {\em Graduate
  Studies in Mathematics}.
\newblock American Mathematical Society, Providence, RI, second edition, 2017.

\bibitem{Maz}
V.~G.~Maz'ja.
\newblock {\em Sobolev spaces}.
\newblock Springer Series in Soviet Mathematics. Springer-Verlag, Berlin, 1985.
\newblock Translated from the Russian by T. O. Shaposhnikova.

\bibitem{MazSha}
V.~G.~Maz'ja and T.~Shaposhnikova.
\newblock On pointwise interpolation inequalities for derivatives.
\newblock {\em Math. Bohem.}, 124(2-3):131--148, 1999.

\bibitem{McCRoRo}
D.~S.~McCormick, J.~C.~Robinson and J.~L.~Rodrigo.
\newblock Generalised Gagliardo-Nirenberg inequalities using weak Lebesgue spaces and BMO.
\newblock {\em  Milan J. Math.}, 81(2):265--289, 2013.

\bibitem{Mirandalinc}
C.~Miranda.
\newblock Su alcune diseguaglianze integrali.
\newblock {\em Atti Accad. Naz. Lincei Mem. Cl. Sci. Fis. Mat. Natur. Sez. I
  (8)}, 7:1--14, 1963.
 
  \bibitem{MirIta}
C.~Miranda.
\newblock {\em Istituzioni di Analisi Funzionale lineare, Vol. I-II}.
\newblock Unione Matematica Italiana and C.N.R., Tip. Oderisi Ed., Gubbio (Italy), 1978.

\bibitem{Nire1}
L.~Nirenberg.
\newblock On elliptic partial differential equations.
\newblock {\em Ann. Scuola Norm. Sup. Pisa (3)}, 13:115--162, 1959.

\bibitem{Nire2}
L.~Nirenberg.
\newblock An extended interpolation inequality.
\newblock {\em Ann. Scuola Norm. Sup. Pisa (3)}, 20:733--737, 1966.

\bibitem{MRS}
A.~Molchanova, T.~Roskovec and F.~Soudsk{\`y}.
\newblock Interpolation between {H}\" older and {L}ebesgue spaces with
  applications.
\newblock {\em arXiv preprint arXiv:1801.06865}, 2018.

\bibitem{Solon}
V.~A.~Solonnikov.
\newblock On certain inequalities for functions belonging to $\vec
  w_{p}(r^n)$-classes.
\newblock {\em Boundary-value problems of mathematical physics and related
  problems of function theory, Part~6. Zap. Nauchn. Sem. LOMI}, 27:194--210,
  1972.
  
\bibitem{Strz}
P.~Strzelecki.
\newblock Gagliardo--Nirenberg inequalities with a BMO term.
\newblock {\em Bulletin of the London Mathematical Society},  38(2): 294--300,
  2006.  

\bibitem{Tor}
A.~Torchinsky.
\newblock {\em Real-variable methods in Harmonic Analysis}, volume 123 of {\em Pure and Applied Mathematics
  (London)}.
\newblock Academic Press, 1986.

\bibitem{Triebel}
H.~Triebel.
\newblock Gagliardo-{N}irenberg inequalities.
\newblock {\em Proc. Steklov Inst. Math.}, 284(1):263--279, 2014.

\bibitem{Zie}
W.~P.~Ziemer.
\newblock {\em Weakly differentiable functions}, volume 120 of {\em Graduate
  Texts in Mathematics}.
\newblock Springer-Verlag, New York, 1989.
\newblock Sobolev spaces and functions of bounded variation.

\end{thebibliography}

\end{document}